\documentclass[]{amsart}

\usepackage{graphicx, xcolor}
\usepackage{amssymb}
\usepackage{subfigure}
\usepackage{caption}
\usepackage{amsfonts}
\usepackage{amsmath,amsthm}
\usepackage{enumerate}
\newcommand{\Hil}{\mathcal{H}}

\newcommand{\B}{\mathcal{B}}

\newcommand{\N}{\mathbb{N}}

\newcommand{\R}{\mathbb{R}}
\newcommand{\K}{\mathbb{K}}

\newcommand{\C}{\mathbb{C}}

\newcommand{\lsub}{{\lambda \in \Lambda}}
\newcommand{\BL}{{\mathfrak{B}}}
\newcommand{\betap}{\beta_{\Phi,\B,\BL}}
\newcommand{\Bp}{B_{\Phi,\B,\BL}}
\newcommand{\Ap}{A_{\Phi,\B,\BL}}

\newcommand{\eps}{\varepsilon}
\newcommand{\supp}{\text{supp} \,}
\newcommand{\sinc}{\text{sinc} \,}

\newtheorem{theorem}{Theorem}[section]
\newtheorem{remark}{Remark}[section]

\newtheorem{definition}[theorem]{Definition}

\newtheorem{lemma}[theorem]{Lemma}
\newtheorem{corollary}[theorem]{Corollary}

\newtheorem{problem}[theorem]{Problem}
\newtheorem{example}[theorem]{Example}

\title[Phase retrieval for continuous frames in Banach spaces]{Phase retrieval in the general setting of continuous frames for Banach spaces}
\author{Rima Alaifari}
\address[Rima Alaifari]{ Department of Mathematics, ETH Z\"{u}rich, R\"{a}mistrasse 101, 8092 Z\"{u}rich, Switzerland}
\author{Philipp Grohs}
\address[Philipp Grohs]{ Department of Mathematics, ETH Z\"{u}rich, R\"{a}mistrasse 101, 8092 Z\"{u}rich, Switzerland and
Faculty of Mathematics, University of Vienna, Oskar-Morgenstern-Platz 1, 1090 Wien, Austria}

\begin{document}

\date{\today}
\maketitle

\begin{abstract}

We develop a novel and unifying setting for phase retrieval problems that works in Banach spaces and for continuous frames and consider the questions of uniqueness and stability of the reconstruction from phaseless measurements. Our main result states that also in this framework, the problem of phase retrieval is never uniformly stable in infinite dimensions. On the other hand, we  show weak stability of the problem. This complements recent work \cite{cahill}, where it has been shown that phase retrieval is always unstable for the setting of discrete frames in Hilbert spaces. In particular, our result implies that the stability properties cannot be improved by oversampling the underlying discrete frame.

We generalize the notion of \emph{complement property (CP)} to the setting of continuous frames for Banach spaces (over $\K=\R$ or $\K=\C$) and verify that it is a necessary condition for uniqueness of the phase retrieval problem; when $\K=\R$ the CP is also sufficient for uniqueness. In our general setting, we also prove a conjecture posed by Bandeira et al. \cite{Bandeira_SavingPhase}, which was originally formulated for finite-dimensional spaces: for the case $\K=\C$ the \emph{strong complement property (SCP)} is a necessary condition for stability. To prove our main result, we show that the SCP can never hold for frames of infinite-dimensional Banach spaces.

\end{abstract}

\section{Introduction}

In different applications, where one seeks to reconstruct an unknown function $f$, only measurements that do not contain phase (or sign) information are at hand. Phase retrieval deals with the recovery of $f$ from such \emph{phaseless} intensity measurements.

The most prominent example of a phase retrieval problem was introduced with the discovery of X-ray crystallography  \cite{miao2008extending} in 1915 to determine the structure of crystals: Recover a function $f$ from the magnitude of its Fourier transform, $|\hat{f}|$. This problem arises also in other applications, such as in diffractive imaging \cite{bunk2007diffractive} and in astronomy \cite{fienup1987phase}. It is known that the solution to this problem is in general not unique.

In more general phase retrieval problems, one seeks to reconstruct a function $f$ in a Hilbert space $\Hil$ from the magnitudes of its \emph{frame coefficients,} i.e. from $(|\langle f, \varphi_\lambda \rangle|)_{\lambda \in \Lambda}$ for a \emph{frame} $\Phi = (\varphi_\lambda)_{\lambda \in \Lambda}$ of $\Hil$. Such problems occur in audio processing applications, in particular in signal noise reduction \cite{deller1993discrete} and automatic speech recognition \cite{becchetti2008speech}. In these applications the signals typically have a highly noisy phase so that it is essential to employ \emph{phaseless} measurements only.

In the finite-dimensional setting, this problem 
and its properties have been studied in e.g. \cite{balan2015reconstruction, balan2006signal, balan2015lipschitz, Bandeira_SavingPhase, candes2015phase}. More precisely, in these works it has been analyzed under which conditions the problem is uniquely solvable, and in case of injectivity, whether the reconstruction from phaseless measurements is stable in a certain sense.

In infinite dimensions and for a specific frame, namely a \emph{semi-discrete frame} of Cauchy wavelets, Mallat \& Waldspurger \cite{mw} have been able to show injectivity and weak stability of the phase retrieval problem. Moreover, they also state that the problem is not uniformly stable. Together with Daubechies \& Thakur \cite{adgt}, we prove uniqueness and weak stability of the phase retrieval problem for a general class of semi-discrete frames, given that they consist of real-valued and band-limited functions.

Recently, Cahill, Casazza \& Daubechies \cite{cahill} considered the general question of stability of the reconstruction procedure in the setting of Hilbert spaces and discrete frames. Their main findings are that the problem of phase retrieval is always stable for finite-dimensional Hilbert spaces but can never be uniformly stable in the infinite dimensional setting. Moreover they study the degradation of stability in finite dimensional subspaces as the dimensions of these subspaces grow. For a specific instance, the stability is shown to degrade at an exponential rate in the dimensions of the chosen subspaces. We briefly describe their example:

\begin{example}[Cahill, Casazza, Daubechies]\label{ex:Cahill}
For the Hilbert space
$$
	\Hil = \{ f \in L^2(\mathbb{R},\mathbb{R}): \supp \widehat{f} \subseteq [-\pi,\pi]\},
$$
consider the frame $\Phi = \{ \varphi_n\}_{n \in \mathbb{Z}}$ of elements
$\varphi_n := \sinc( \cdot - \frac{n}{4})$. Let $V_n \subset \Hil$ be defined as 
$$
	V_n := \text{span } \{ \varphi_{4 \ell}: \ell \in [-n,n]\}.
$$
Then, there exists a constant $C>0$ such that for all $m\in \mathbb{N}$, there exist $f_m, g_m \in V_{2m}$ with
$$
	\min_{\tau \in \{ \pm 1\}} \|f_m - \tau g_m \|_\Hil > C (m+1)^{-1} 2^{3m} \| (|f_m(n/4)|-|g_m(n/4)|)_{n \in \mathbb{Z}} \|_{\ell^2(\mathbb{Z})}.
$$
\end{example}
A natural question to ask is whether in this example, increasing the redundancy of the frame $\Phi$ could improve on the stability of the problem, resulting in bounds that deteriorate less severely in the dimension $m$ of the subspace. 

Motivated by this question, we consider phase retrieval in a general and unifying setting of continuous frames in Banach spaces in this paper. The main reason for studying general continuous frames is the question whether suitably \emph{oversampling} a discrete frame can improve on the stability properties of the phase retrieval problem. In addition, it provides a unified setting that also includes the cases of semi-discrete frames studied in earlier works. Choosing to work in Banach spaces allows us to treat phase retrieval problems that do not fit in the Hilbert space setting, such as the reconstruction of real-valued band-limited functions in general Paley-Wiener spaces \cite{thakur}.

Before providing a summary of our main findings in Section \ref{contributions}, we give a brief description of the mathematical setting used throughout this paper.

%


\subsection{Mathematical Setup}
First, we fix a reflexive Banach space $\B$ over the field $\K$, where $\K=\C$ or $\K=\R$. Let $\B'$ denote the (topological) dual space of $\B$, i.e., the Banach space of bounded linear functionals $\psi:\B \to \K$. For $x \in \B$, $\psi \in \B'$ we write $[x,\psi]:=\psi(x).$ For a closed subspace $V\subset \mathcal{B}$ we denote $V_\bot:=\{\psi \in \mathcal{B}':\ [x,\psi]=0,\ \mbox{for all }x\in V\}$ the annihilator of $V$ and we use the same notation for the annihilator of a closed  subspace of $\mathcal{B}'$. Note that for the special case when $\B=\mathcal{H}$, $\mathcal{H}$ a Hilbert space with inner product $\langle \cdot, \cdot \rangle_{\mathcal{H}}$, the Riesz lemma implies that $\mathcal{H}$ can be identified with $\mathcal{H}'$ and hence, $[\cdot, \cdot]$ with $ \langle \cdot, \cdot \rangle_{\mathcal{H}}$. 

We further fix a representation system $\Phi=(\varphi_\lambda)_{\lambda \in \Lambda} \subset \mathcal{B}'$ for some (not necessarily discrete) index set $\Lambda$ and the notation $\Phi_S:=(\varphi_\lambda)_{\lambda \in S}$ for $S \subset \Lambda$. For $x \in \B$, we introduce the following abbreviations:
$$
	[x,\Phi]:= ([x,\varphi_\lambda])_\lsub, \quad [x,\Phi_S]:= ([x,\varphi_\lambda])_{\lambda \in S}
$$
and
$$
	|[x,\Phi]|:= (|[x,\varphi_\lambda]|)_\lsub, \quad |[x,\Phi_S]|:= (|[x,\varphi_\lambda]|)_{\lambda \in S}.
$$

We define the quotient space $\mathbb{P}_{\K}\mathcal{B}$ of $\B$ by $\mathbb{P}_{\R}\mathcal{B}:=\mathcal{B}/\{1,-1\}$ for $\K=\mathbb{R}$ and by $\mathbb{P}_{\C}\mathcal{B}:=\mathcal{B}/\mathcal{S}^1$ for $\K=\C$ with distance
$$
	d_\mathcal{B}(x,y):=\min\{\|x-\tau y\|_\B; \tau \in \K, |\tau|=1\}.
$$

The problem of phase retrieval that we consider throughout this paper can now be formulated as follows:
\begin{problem}\label{def-problem}
Define the operator 
\begin{equation}\label{eq:defA}
	\mathcal{A}_\Phi:\mathbb{P}_{\K}\mathcal{B}\to \R_+^\Lambda,\quad
	x\mapsto |[x,\Phi]|.
\end{equation}
Then, the problem of phase retrieval can be formulated as the inversion of $\mathcal{A}_\Phi$. We seek to study the uniqueness and stability properties of this inversion problem. 
\end{problem}
\subsection{Contributions}\label{contributions}
We give a brief overview of the main results of our paper:
\begin{itemize}
%

\item One of the major contributions (Theorem \ref{nostab}) of our paper is a generalization of the results of \cite{cahill} to the non-discrete Banach space setting. In particular, our work for the first time proves the fact that oversampling does not improve the stability of phase retrieval and further confirms the fundamental instability of this problem. \\

\item As part of the proof of Theorem \ref{nostab} we establish the auxiliary fact (Theorem \ref{noframe}) that the index set for a continuous frame is necessarily non-compact for an infinite-dimensional Banach space,
which by itself is an interesting finding. \\

\item On the positive side we are able to show in Theorem \ref{thm:weakstab} that the reconstruction is
always locally stable which, by a compactness argument, directly implies global stability
in the finite-dimensional setting. Local stability has been shown in \cite{mw} in a very specific setting and we adapt and generalize the arguments there to our general setting. To this end we require a novel characterization of relative compactness in solid Banach spaces (Theorem \ref{thm:totalbound}) which is interesting in its own right. \\

\item As another contribution (Theorem \ref{thmstab}) we establish the necessity of the so-called
\emph{strong complement property} introduced in \cite{Bandeira_SavingPhase} for stable phase reconstruction -- a fact which, in the complex case, was only conjectured in \cite{Bandeira_SavingPhase}. This finding suggests that phase retrieval in the complex-valued setting is at least as ill-posed as in the real case.
\end{itemize}
\subsection{Outline} The paper is organized as follows. In Section \ref{sec:injectivity}, the question of uniqueness is considered. It is shown, that as in the finite dimensional setting, the complement property is again a necessary condition for uniqueness and that it is also a sufficient condition in the real-valued case. Section \ref{sec:stability} concerns the question of stability. First, the concepts of admissible Banach spaces and Banach frames are introduced. In Section \ref{sec:weak-stability} we show a weak stability result. Next, we consider uniform stability in Section \ref{sec:strong-stability}. There, the strong complement property is defined and it is proved that it provides a necessary condition for stability. Moreover, we show that it is also a sufficient condition for stability in the real-valued case. In \ref{sec:finite}, uniform stability is derived for the finite dimensional case. Finally, in \ref{sec:infinite} we prove that the strong complement property can never hold in the infinite dimensional setting, so that consequently, the problem of phase retrieval is always unstable in infinite-dimensional Banach spaces.

\section{Injectivity}\label{sec:injectivity}
We first consider the injectivity of $\mathcal{A}_\Phi$.
To this end the following property has been introduced in \cite{balan2006signal}.

\begin{definition}[Complement property]
	The system $\Phi=(\varphi_\lambda)_{\lambda \in \Lambda} \subset \mathcal{B}'$ satisfies the \emph{complement property} (CP) in $\mathcal{B}$ if for every subset 
	$S\subset \Lambda$, either $(\mbox{span }\Phi_S)_\bot=\{0\}$, or $(\mbox{span }\Phi_{\Lambda\setminus S})_\bot=\{0\}$.
\end{definition}

The following two lemmata have been shown in \cite{balan2006signal} and \cite{Bandeira_SavingPhase} in the finite-dimensional setting, as well as in \cite{cahill} for the setting of discrete frames of Hilbert spaces. The proofs carry over to the Banach space case and for uncountable index sets $\Lambda$, but we present them nevertheless, for convenience and to be certain.
\begin{lemma}\label{CPlem}
	If $\mathcal{A}_\Phi$ is injective, then $\Phi$ satisfies the CP in $\B$.
\end{lemma}
\begin{proof}
	Suppose that $\mathcal{A}_\Phi$ is injective and consider $u\in (\mbox{span }\Phi_S)_\bot$ and
	$v\in (\mbox{span }\Phi_{\Lambda\setminus S})_\bot$. 
	Then, 
	$$
		|[ u\pm v,\varphi_\lambda]|^2
		=|[u,\varphi_\lambda]|^2
		\pm 2 \, \text{Re} ([u,\varphi_\lambda][v,\varphi_\lambda])
		+|[ v,\varphi_\lambda]|^2.
	$$
	By assumption the term $[u,\varphi_\lambda][v,\varphi_\lambda]$ vanishes for all $\lambda \in \Lambda$, which implies that $\mathcal{A}_\Phi(u+v) = \mathcal{A}_\Phi(u-v)$. By the injectivity assumption this implies that either $u=0$ or 
	$$
		v = - \frac{1-\tau}{1+\tau} u
	$$
	for some $\tau \neq -1$. The latter implies $v \in (\mbox{span } \Phi_S)_\bot$ and consequently,
	$[v,\varphi_\lambda]=0$ for all $\lambda \in \Lambda$. As a result, $\mathcal{A}_\Phi(v)=0$ and by the injectivity of $\mathcal{A}_\Phi$, $v=0$. Thus, either $u=0$ or $v=0$, which is precisely the CP.
\end{proof}
In the real case, $\K=\R$, the CP is also a sufficient condition for the injectivity of $\mathcal{A}_\Phi$:
\begin{lemma}\label{lem:CPsuff}
If $\B$ is a Banach space over $\R$, then $\mathcal{A}_\Phi$ is injective if and only if $\Phi$ satisfies the CP in $\B$.
\end{lemma}
\begin{proof}
By Lemma \ref{CPlem}, the injectivity of $\mathcal{A}_\Phi$ implies that $\Phi$ satisfies the CP. For the other direction, suppose now that $\Phi$ satisfies the CP and
	that $x,y\in \B$ with $\mathcal{A}_\Phi(x) = \mathcal{A}_\Phi(y)$. Set $S:=\{\lambda \in \Lambda:\ [x, \varphi_\lambda]=
	[y, \varphi_\lambda]\}$. It follows that $x-y\in (\mbox{span }\Phi_S)_\bot$ and $x+y\in (\mbox{span }\Phi_{\Lambda\setminus S})_\bot$. By the CP it follows that $d_\mathcal{B}(x,y)=0$.
\end{proof}

\subsection{An example}\label{ex:Beurling-CP}
As an application for showing the injectivity by employing the CP, we consider
the reconstruction in the Paley-Wiener space

$$
	\mathcal{PW}^{p,b}_\R:=
	\{f \in L^p(\mathbb{R},\mathbb{R}): \supp \widehat{f} \subseteq [-b/2,b/2]\},
$$
with norm 
$$
	\|f\|_{\mathcal{PW}^{p,b}_\mathbb{K}}:=\|f\|_p
$$
and $1<p<\infty$.

Functions in $\mathcal{PW}^{p,b}_\mathbb{R}$
can be uniquely reconstructed from sufficiently fine sequences of their unsigned samples, as we now show (a similar result is shown in \cite{thakur}, yet we want to illustrate that verifying the CP gives a more general result and allows for a simpler proof).

Suppose that $\Lambda \subset \R$. Consider the representation system $\Phi:=(\varphi_\lambda)_{\lambda\in \Lambda}$ defined by
$$
	[f,\varphi_\lambda] := f(\lambda).
$$
We call $\Lambda$ a \emph{sampling sequence} \cite{seip2004interpolation} for $\mathcal{PW}^{p,b}_\mathbb{R}$ if for all $f \in \mathcal{PW}^{p,b}_\mathbb{R}$, the condition
$$
	[f,\varphi_\lambda]= 0  \text{ for all } \lambda \in \Lambda	
$$
implies that $f=0$. 

Various conditions for $\Lambda$ to be sampling (in terms of density properties) are known and a general summary can be found in e.g. \cite{seip2004interpolation}. One way to characterize sampling sequences is via their \emph{lower Beurling density} \cite{carleson1989collected}, which for a sequence $\Lambda$ can be defined as
\begin{align*}
D^-(\Lambda) &:= \liminf_{r \to \infty} \inf_a \frac{\#(\Lambda,[a,a+r))}{r},
\end{align*}
where $\#(\Lambda,[a,a+r))$ denotes the number of elements $\lambda \in \Lambda \cap [a,a+r)$. Sampling sequences can then be characterized as follows \cite{bruna2001sampling}:
\begin{theorem}[E.g. \cite{bruna2001sampling,seip2004interpolation}]
\quad

\begin{itemize}
\item For $p \in (0,1)$ or $p=\infty$, 

	$\Lambda$ is sampling for $\mathcal{PW}^{p,b}_\mathbb{R}$ if and only if $D^-(\Lambda)>b$;
\item for $p \in (1,\infty),$

	$\Lambda$ is sampling for $\mathcal{PW}^{p,b}_\mathbb{R}$ if $D^-(\Lambda)>b$ and only if $D^-(\Lambda) \geq b$.
\end{itemize}
\end{theorem}
We now state our following result that gives a general criterion on $\Lambda$ for injectivity of $\mathcal{A}_\Phi$.

\begin{theorem}
	Suppose that $\Lambda$ is a sampling sequence for 
	$\mathcal{PW}^{p/2,2b}_\R$. Then, 
	$$
		\mathcal{A}_\Phi:\mathcal{PW}^{p,b}_\R/\{1,-1\} \to \R_+^\Lambda
	$$ 
	is injective.
\end{theorem}

\begin{proof}
	By Lemma \ref{lem:CPsuff} we need to show that $\Phi$
	satisfies the CP. Suppose that the CP does not hold in 
	$\mathcal{PW}^{p,b}_\R$. Then there exist $u,v\in 
	\mathcal{PW}^{p,b}_\R\setminus \{0\}$ and $S\subset \Lambda$ with
	$u(S) = 0$ and $v(\Lambda\setminus S)= 0$. 
	Let $h:=u\cdot v\in \mathcal{PW}^{p/2,2b}_\R$. We note that $u,v$ and $h$ are holomorphic functions. If $h$ vanishes on an open interval $\Omega \subset \R$, then there is a subinterval $U \subset \Omega$ on which $u$ has no zeros. Since $u \cdot v = 0$ on $U$, this implies $v=0$ on $U$. Therefore, $v$ has to vanish identically on $\R$, which is a contradiction to our assumption. We can thus deduce that $h \in \mathcal{PW}^{p/2,2b}_\R\setminus \{0\}$. By the properties of $u$ and $v$, we also have $h(\Lambda)=0$, which contradicts the sampling condition.
\end{proof}

\begin{corollary}
If $\Lambda \subset \mathbb{R}$ is a sequence with lower Beurling density  $D^-(\Lambda)>2b$, then 
$$
		\mathcal{A}_\Phi:\mathcal{PW}^{p,b}_\R/\{1,-1\} \to \R_+^\Lambda
$$
is injective.
\end{corollary}

\begin{remark}
We note that the result in \cite{thakur} on the injectivity of $\mathcal{A}_\Phi$ required the sequence $\Lambda$ to have uniform density, which is a stronger condition than $\Lambda$ being a sampling sequence.
\end{remark}
\section{Stability}\label{sec:stability}
Given a representation (or measurement) system $\Phi$ satisfying the CP, it
is interesting to study whether the inversion of $\mathcal{A}_{\Phi}$
is stable.
To this end we need to put a metric structure on elements of $\K^\Lambda$.  We do this by assuming that the measurements
$\mathcal{A}_{\Phi}(x)$, $x\in \mathcal{B}$, always lie in 
 a fixed Banach space $\BL \subset \K^\Lambda$ with norm $\| \cdot \|_{\BL}$. Now we can consider $\mathcal{A}_\Phi$
 as a mapping between two Banach spaces $\mathcal{B}$ and $\BL$ and the question of continuity makes sense.
 
 To avoid pathologies we impose the following weak conditions on
 the space $\BL$. In what follows we shall write
 $\chi_S$ for the indicator function of $S\subset \Lambda$.

\begin{definition}[Admissible Banach space]\label{assump}
	Let $\Lambda$ be a $\sigma$-compact topological space. A Banach space $\BL \subset \K^\Lambda$ over $\K$  is called \emph{admissible} if it satisfies the following properties:
	\begin{itemize}
	\item[(i)]For every compact set $\Lambda'\subset \Lambda$
		the norm of the indicator function is finite, i.e., 
		$$
			\| \chi_{\Lambda'}\|_\BL<\infty.
		$$

		\item[(ii)]$\BL$ is \emph{solid}, i.e., it holds that
		\begin{equation}
			\|w\|_{\BL}=\| |w| \|_{\BL}\le
			\| z\|_{\BL}
		\end{equation}
		for all $z\in\BL$ with $|z(\lambda)| \geq |w(\lambda)|$ for all $\lambda \in \Lambda$.
		\item[(iii)] 
		The set of elements in $\BL$ with compact support is dense in $\BL$.
			\end{itemize}
\end{definition}
Note that the assumptions in Definition \ref{assump} are truly weak: for example they are satisfied by every weighted $L^p$ space.

We record the following useful property that we will use in our proofs and which follows directly from the solidity condition:

\begin{equation}\label{subset-ineq}
\|w\|_{\BL} \geq \|w \chi_{S}\|_{\BL} \text{ for all } w \in \BL \text{ and } S \subset \Lambda.
\end{equation}
 In order to ensure that the image of $\mathcal{B}$ under the measurement mapping $\mathcal{A}_\Phi$ is contained in $\BL$ we assume further that the representation system $\Phi$ constitutes a Banach frame for $\mathcal{B}$, as
 introduced in the following definition.
 \begin{definition}[Banach frame]\label{def:frame}
Let $\Phi:=(\varphi_\lambda)_{\lambda \in \Lambda}\subset \B'$ be a family of bounded linear functionals such that the mapping $\lambda \mapsto \varphi_\lambda$ is continuous.
	Suppose there is an admissible Banach space $\BL \subset \K^{\Lambda}$ with norm $\| \cdot \|_{\BL}$ and with the following properties:
	\begin{itemize}
	\item[(B1)]
	For some real constants $B \geq A>0,$ 
	$$
		A\|x\|_\B \le \|[x,\Phi] \|_{\BL}\le 
		B\|x\|_\B
	$$
	holds for all $x\in \mathcal{B}$.
	\item[(B2)]
	There exists a continuous operator $R:\BL \to \mathcal{B}$ with 
	$$
		R([ x,\Phi]) = x\quad
		\mbox{for all }x\in \mathcal{B}.
	$$
	\end{itemize}
	Then, $\Phi$ is called a \emph{frame} for $\mathcal{B}$. We denote the best possible such constants by
	$A_{\Phi,\mathcal{B},\BL}, \ B_{\Phi,\mathcal{B},\BL}$ and call them the \emph{frame constants} of $\Phi$.
\end{definition}
Again the assumption that $\Phi$ constitutes a frame is very weak; in fact it is necessary for stability of the reconstruction even if 
the phases of $[x,\varphi_\lambda]$ were known for all $\lambda \in \Lambda$.

Having introduced the fundamental concepts of admissible Banach spaces and Banach frames we can now make the question of stability precise:
\begin{quote}
	Suppose that $\Phi$ constitutes a frame as in Definition \ref{def:frame} and that the measurement mapping $\mathcal{A}_\Phi$
	is injective.
	\begin{itemize}
		\item[(Q1)] Is the mapping $\mathcal{A}_\Phi:\mathcal{B}\to \BL$ continuously invertible on its range?
		\item[(Q2)]Is the mapping $\mathcal{A}_\Phi:\mathcal{B}\to \BL$ uniformly continuously invertible on its range?
	\end{itemize}
\end{quote}
In the remainder of this section we will show that
\begin{itemize}
\item The answer to (Q1) is always yes.
\item The answer to (Q2) is yes if and only if $\mathcal{B}$ is finite dimensional.
\end{itemize}
In addition we will also obtain some quantitative estimates on the Lipschitz constants of $\mathcal{A}_\Phi^{-1}$.
\subsection{Weak Stability}\label{sec:weak-stability}
 
We first show that the inverse $\mathcal{A}_\Phi^{-1}$ is always continuous, 
whenever it exists and thereby answer (Q1) affirmatively.
\begin{theorem}\label{thm:weakstab}
	Suppose that $\Phi$ constitutes a frame for $\mathcal{B}$ with associated admissible Banach space $\BL$, and that $\mathcal{A}_\Phi$ is injective. Then the 
	operator $\mathcal{A}_\Phi^{-1}$ is continuous on its range. 
\end{theorem}
\begin{proof}
	We need to show that for every Cauchy sequence $(\mathcal{A}_\Phi(x_k))_{k\in \N}$ in $\BL$ the sequence
	$(x_k)_{k \in \N}$ converges in $\mathcal{B}$. If  $(\mathcal{A}_\Phi(x_k))_{k\in \N}$ is a Cauchy sequence, the set
	$|\mathfrak{K}|:=\{|[x_k,\Phi]|:\ k\in \N\}$ is relatively compact in
	$\BL$. Since $\BL$ is a Banach space, relative compactness in $\BL$ is equivalent to total boundedness in $\BL$. Hence, $|\mathfrak{K}|$ is bounded, which, by the frame property, implies that $\{x_k:\ k\in \N\}$ is bounded in $\mathcal{B}$.
	This, together with the continuity of the mapping $\lambda\mapsto \varphi_\lambda$, implies the strong equicontinuity (see Def. \ref{def:equicont}) of
	the family $\mathfrak{K}:=\{[x_k,\Phi]:\ k\in \N\}$, which can be seen as follows:
	
	By the boundedness of $\{x_k:\ k\in \N\}$, there exists a constant $c>0$, such that for all $k \in \N$
	$$
		\| x_k \|_\B \leq c.
	$$
	The continuity of the mapping $\lambda\mapsto \varphi_\lambda$ yields that for all $\eps >0$ there is a neighborhood $U_\eps(\lambda) \subset \Lambda$ such that
	$$
		\| \varphi_\lambda-\varphi_\mu \|_{\B'} \leq \frac{\eps}{c}
	$$
	for all $\mu \in U_\eps(\lambda)$.
	Therefore,
	$$
		\sup_{k \in \N} |[x_k,\varphi_\lambda]-[x_k,\varphi_\mu]| = \sup_{k \in \N} |[x_k,\varphi_\lambda-\varphi_\mu]|\leq c \| \varphi_\lambda - \varphi_\mu \|_{\B'} \leq \eps
	$$
	for all $\mu \in U_\eps(\lambda)$, which shows that $\mathfrak{K}$ is strongly equicontinuous.
	
	Now we can use Corollary \ref{cor:relcomp} and the relative compactness of $|\mathfrak{K}|$ to conclude that
	the set
	$$
		\mathfrak{K}:=\{ [x_k,\Phi]:\ k\in\N\}
	$$
	is relatively compact in $\BL$. 
	Thus, there exists a convergent subsequence 
	$$
		([ x_{k_l},\Phi])_{l\in\N} 
	$$
	in $\BL$.
	
	Together with the frame property this implies that the sequence
	$$
		x_{k_l}=R([x_{k_l},\Phi])
	$$
	is a Cauchy sequence in $\mathcal{B}$, converging to some $x\in \mathcal{B}$.
	Since $\mathcal{A}_\Phi$ is continuous, it follows that 
	$$
		\lim_{l\to\infty}\mathcal{A}_\Phi(x_{k_l})=\mathcal{A}_\Phi(x).
	$$
	Suppose we choose another convergent subsequence converging to a different $x'$. Then, by the injectivity of $\mathcal{A}_\Phi$ we have
	that $x'=x$. Consequently, the full sequence $(x_k)_{k \in \mathbb{N}}$ converges.
\end{proof}
\subsection{Strong Stability}\label{sec:strong-stability}
The continuity result which we have just shown is mainly of theoretical interest, due to the fact that it is non-quantitative. 
In order to obtain quantitative results we 
look for the greatest possible $\alpha>0$ and the
smallest possible $\beta<\infty$ for which
$$
	\alpha d_\mathcal{B}(x,y)\le \left\|\mathcal{A}_\Phi(x) - \mathcal{A}_\Phi(y)\right\|_{\BL}
	\le \beta d_\mathcal{B} (x,y).
$$
We call the best possible such constants $\alpha_{\Phi,\mathcal{B},\BL},\  \beta_{\Phi,\mathcal{B},\BL}$.

Clearly, the stability of the inversion of $\mathcal{A}_\Phi$
is determined by the ratio of these constants.
\begin{definition}[Condition number]
	We call the constant
	$$
		\tau_{\Phi,\mathcal{B},\BL}:=\frac{\beta_{\Phi,\mathcal{B},\BL}}{\alpha_{\Phi,\mathcal{B},\BL}}
	$$
	the \emph{condition number} of $\mathcal{A}_\Phi$ related to $\mathcal{B}, \BL$.
\end{definition}
As a first step we show that the upper Lipschitz constant $\beta_{\Phi,\mathcal{B},\BL}$ is precisely given by the
upper frame constant of $\Phi$.
\begin{theorem}\label{stabup-real}
	Suppose that $\Phi$ constitutes a frame for $\mathcal{B}$. 
	Then, 
	$$
		\beta_{\Phi,\mathcal{B},\BL}=B_{\Phi,\mathcal{B},\BL}.
	$$
\end{theorem}
\begin{proof}
		By the solidity of $\BL$ and using the reverse triangle inequality, we have
		\begin{align*}
			\|\mathcal{A}_\Phi(x)-\mathcal{A}_\Phi(y)\|_{\BL}
			&= \| (|[ x,\varphi_\lambda]|-|[y,\varphi_\lambda]|)_{\lambda\in\Lambda}\|_{\BL}\\
			&\le 
			\|(\min\{|[x-\tau y,\varphi_\lambda]|: |\tau|=1\})_{\lambda\in\Lambda}\|_{\BL}\\
			&\le 
			\min\{\|(|[x-\tau y,\varphi_\lambda]|)_{\lambda\in\Lambda}\|_{\BL}:|\tau|=1\}\\
			&\le 
			B_{\Phi,\mathcal{B},\BL} d_{\mathcal{B}}(x,y),
		\end{align*}
		where the last inequality follows from the frame property. Hence, $\betap \leq \Bp$. The converse inequality follows from taking $x\in \mathcal{B}$ which saturates
		the upper frame bound and $y=0$.
\end{proof}
In order to study the lower Lipschitz bound $\alpha_{\Phi,\mathcal{B},\BL}$ the following quantitative version of the CP has been introduced in \cite{Bandeira_SavingPhase} in the finite dimensional real-valued setting.
\begin{definition}[Strong complement property]
	The system $\Phi = (\varphi_\lambda)_{\lambda \in \Lambda} \subset \B'$ satisfies the \emph{$\sigma$-strong complement property} (SCP) in $\mathcal{B}$ if there exists $\sigma >0$ such that for every subset $S\subset \Lambda$, we have
	$$
		\max\{A_{\Phi_S,\mathcal{B},\BL},A_{\Phi_{\Lambda\setminus S},\mathcal{B},\BL}\}\geq \sigma.
	$$
	We denote the supremal $\sigma$ for which $\Phi$ satisfies the SCP by $\sigma_{\Phi,\mathcal{B},\BL}$.
\end{definition}
In what follows we show that if $\Phi$ constitutes a frame, then in the complex-valued setting (i.e., when $\K=\C$) the SCP is a necessary condition for the stability of Problem \ref{def-problem} (i.e., for the condition $\tau_{\Phi,\mathcal{B},\BL}<\infty$). This result can be strengthened in the real-valued setting ($\K=\R$). In this case we prove that the SCP is also a sufficient condition for the stability of Problem \ref{def-problem}. 
\subsubsection{The real case}
We first formulate the simpler result in the real-valued setting for which the proof roughly follows the lines of \cite{Bandeira_SavingPhase}.
\begin{theorem}\label{stablow-real}
	Suppose that $\B$ is a Banach space over $\R$ and that $\Phi = (\varphi_\lambda)_{\lambda \in \Lambda} \subset \B'$ constitutes a frame of $\mathcal{B}$. Then
	$$
		\sigma_{\Phi,\mathcal{B},\BL} \le \alpha_{\Phi,\mathcal{B},\BL} \le 2\sigma_{\Phi,\mathcal{B},\BL}.
	$$
\end{theorem}
\begin{proof}
	We first show that $\alpha_{\Phi,\mathcal{B},\BL} \le 2\sigma_{\Phi,\mathcal{B},\BL}$. Let $\sigma > \sigma_{\Phi,\mathcal{B},\BL}$.
		There exist $S\subset \Lambda$ and $u,v\in \mathcal{B}$ with $\|u\|_{\mathcal{B}}=1=\|v\|_{\mathcal{B}}$ such that
		$$
			 \|[u,\Phi_S] \|_{\BL} \leq \sigma
		$$
		and
		$$
			\|[v,\Phi_{\Lambda\setminus S}] \|_{\BL} \leq \sigma.
		$$
		Put $x:=u+v$ and $y:=u-v$. It holds that 
		\begin{align*}
			\|\mathcal{A}_\Phi(x) - \mathcal{A}_\Phi(y)\|_{\BL}
			&\le 
			 \||[x,\Phi_S]| - |[y,\Phi_S]|\|_{\BL}
			 +\||[x,\Phi_{\Lambda\setminus S}]| - |[y,\Phi_{\Lambda\setminus S}]|\|_{\BL}
			 \\
			 &\le  2 \|[u,\Phi_S]\|_{\BL}
			 + 2 \|[v,\Phi_{\Lambda\setminus S}]\|_{\BL},
%
		\end{align*}
		where the last inequality follows from the reverse triangle inequality (cf. \cite{Bandeira_SavingPhase}) and the solidity of $\BL$.
		By our assumptions on $u,v$ this implies that
		$$
			\|\mathcal{A}_\Phi(x) - \mathcal{A}_\Phi(y)\|_{\BL}\le 2\sigma (\|u\|_\B+\|v\|_\B)
			\le 4\sigma \min\{\|u\|_\B,\|v\|_\B\}
			= 2\sigma d_\mathcal{B}(x,y).
		$$
		Since $\sigma>\sigma_{\Phi,\mathcal{B},\BL}$ was arbitrary the desired inequality is proven.
		
		It remains to show the inequality $\sigma_{\Phi,\mathcal{B},\BL} \le \alpha_{\Phi,\mathcal{B},\BL},$ for which we consider $\alpha >\alpha_{\Phi,\mathcal{B},\BL}$. By the optimality
		of $\alpha_{\Phi,\mathcal{B},\BL}$, there exist
		$x,y\in \mathcal{B}$ with 
		$$
			\alpha d_\mathcal{B}(x,y)>
			\|\mathcal{A}_\Phi(x)-\mathcal{A}_\Phi(y)\|_{\BL}.
		$$
		Now pick
		$$
			S:=\{\lambda \in \Lambda:\ \mbox{sign}([x,\varphi_\lambda]) = -
			\mbox{sign}([ y,\varphi_\lambda])\}.
		$$
		Define $u:=x+y$ and $v:=x-y$. 
		It holds that
		\begin{align*}
			\|[u,\Phi_S]\|_{\BL}& = 
			\||[x,\Phi_S]|-|[y,\Phi_S]|\|_{\BL}
			\\
			&\le  \|\mathcal{A}_\Phi(x)-\mathcal{A}_\Phi(y)\|_{\BL}\\
			&\le  
			\alpha d_{\mathcal{B}}(x,y) \le \alpha \|u\|_\B,
		\end{align*}
		where the first inequality follows from (\ref{subset-ineq}). %
		The same argument can be carried out for $v$ and $\Lambda\setminus S$. This implies 
		$$
		\max \{ A_{\Phi_S,\B,\BL}, A_{\Phi_{\Lambda\setminus S},\B,\BL}\} \leq \alpha,
		$$
		which completes the proof. 
\end{proof}
As a corollary we obtain a complete description of the condition number in the real case.
\begin{corollary}\label{cor:stab-real-case}
	Suppose that $\Phi$ constitutes a frame for $\mathcal{B}$
	and that $\K=\mathbb{R}$. 
	Then 
	$$
		\frac{B_{\Phi,\mathcal{B},\BL}}{2\sigma_{\Phi,\mathcal{B},\BL}}
		\le 
		\tau_{\Phi,\mathcal{B},\BL}
		\le 
		\frac{B_{\Phi,\mathcal{B},\BL}}{\sigma_{\Phi,\mathcal{B},\BL}}.
	$$
\end{corollary}
\begin{proof}
	The statement is a direct consequence of Theorems  \ref{stabup-real} and
	\ref{stablow-real}.
\end{proof}
%

\subsubsection{The complex case}
We go on to study the complex case. Naturally the analysis becomes much more complicated in this case and it is in particular impossible to completely characterize the condition number in terms of the SCP. Nevertheless it is possible to provide a lower bound. We note that the following result (Theorem \ref{thmstab}) settles a conjecture posed by Bandeira et al. in \cite{Bandeira_SavingPhase}, where the authors state that they \emph{''suspect that the strong complement property is necessary for stability in the complex case, but have no proof of this''}.
\begin{theorem}\label{thmstab}
Let $\Phi = (\varphi_\lambda)_{\lambda \in \Lambda} \subset \B'$ constitute a frame for $\B$. Then, there exists a constant $C>0$ (depending on only the frame constants of $\Phi$) such that
$$\alpha_{\Phi,\B,\BL} \le C \sigma_{\Phi,\B,\BL}.$$
The constant $C$ can be taken to be
$$
	C = 2 \, \frac{B_{\Phi,\mathcal{B},\BL}}{A_{\Phi,\mathcal{B},\BL}}.
$$
\end{theorem}

\begin{proof}
Let $\sigma > \sigma_{\Phi,\mathcal{B},\BL}$.
		Then, as in the real case, there exist $S\subset \Lambda$ and $u,v\in \mathcal{B}$ with $\|u\|_\B = 1 = \|v\|_\B$ such that
		\begin{equation}\label{instab1comp}
			\|[u,\Phi_S] \|_{\BL} \leq \sigma
		\end{equation}
		and
		\begin{equation}
			 \|[v,\Phi_{\Lambda \setminus S}] \|_{\BL} \leq \sigma.
		\end{equation}
		%
		We also note that, due to the fact that $\Phi$ is a frame, it holds that
		\begin{equation}
			\| [u,\Phi_{\Lambda \setminus S}]\|_{\BL} \geq A_{\Phi, \B, \BL}-\sigma,
		\end{equation}
		and
		\begin{equation}\label{instab2comp2}
			\| [v,\Phi_S]\|_{\BL} \geq A_{\Phi, \B, \BL}-\sigma.
		\end{equation}
		By putting $x:=u+v$ and $y:=u-v$, one again obtains
		\begin{align*}
			\|\mathcal{A}_\Phi(x) - \mathcal{A}_\Phi(y)\|_{\BL}
			 &\le  2 \|[u,\Phi_S]\|_{\BL}
			 + 2
			 \|[v,\Phi_{\Lambda \setminus S}]\|_{\BL}.
		\end{align*}
		%
		By our assumption on $u,v$ this implies that
		$$
			\|\mathcal{A}_\Phi(x) - \mathcal{A}_\Phi(y)\|_{\BL}\le 4\sigma
		$$
		Since $\sigma>\sigma_{\Phi,\mathcal{B},\BL}$, it follows that
		$$
			\alpha_{\Phi,\B,\BL} \le \frac{4 \sigma_{\Phi,\B,\BL}}{d_\B(x,y)}.
		$$
		It remains to show that $d_\B(x,y)$ is not small. We need to give a lower bound on $\| x - \tau y\|_\B$ that holds for any $\tau \in \C$ with $|\tau|=1$. By the frame property of $\Phi$ we have 
		\begin{align*}
	B_{\Phi,\B,\BL}\|x-\tau y\|_\B &\geq 
	\| [x-\tau y, \Phi]\|_\BL. 
	\end{align*}
	Applying (\ref{subset-ineq}) twice for $\Lambda'=S$ and $\Lambda'=\Lambda \setminus S$ yields
	\begin{align*}
		B_{\Phi,\B,\BL}\|x-\tau y\|_\B &\geq \| [x-\tau y,\Phi] \|_\BL \\
		&\geq \frac{1}{2}\| [x-\tau y,\Phi_S] \|_\BL +  \frac{1}{2} \| [x-\tau y,\Phi_{\Lambda \setminus S}] \|_\BL \\
		&\geq \frac{1}{2} \| [(1-\tau) u + (1+\tau) v, \Phi_S] \|_\BL \\
	&+\frac{1}{2}
	\| [(1-\tau) u + (1+\tau) v, \Phi_{\Lambda \setminus S}] \|_\BL.
	\end{align*}
	Employing the reverse triangle inequality and Eqns (\ref{instab1comp})--(\ref{instab2comp2}) we further obtain
	\begin{align*}
	B_{\Phi,\B,\BL}\|x-\tau y\|_\B &\geq
	\frac{|1+\tau|}{2} \|[v,\Phi_S]\|_\BL
	-
	\frac{|1-\tau|}{2}\| [u,\Phi_S]\|_\BL
	\\
	&+
	\frac{|1-\tau|}{2}
	\| [u,\Phi_{\Lambda \setminus S}]\|_\BL
	-
	\frac{|1+\tau|}{2}\| [v,\Phi_{\Lambda \setminus S}]\|_\BL\\
	&\geq 
	\frac{|1+\tau|+|1-\tau|}{2} \cdot \left(A_{\Phi,\mathcal{B},\BL} - 
	2 \sigma\right)
	\\
	&\geq \sqrt{2}(\Ap -  2 \sigma),
	\end{align*}
	provided that $\Ap-2\sigma\geq0$. Thus, we have obtained a lower bound on $\| x-\tau y\|_\B$ which proves that $d_\mathcal{B}(x,y)$ is not small if $\sigma$ is small.
	In particular, we have
	$$
		\alpha_{\Phi,\mathcal{B},\BL}
		\le 
		\sigma_{\Phi,\mathcal{B},\BL} \frac{4 \cdot B_{\Phi,\mathcal{B},\BL}}{\sqrt{2}(A_{\Phi,\mathcal{B},\BL}-2\sigma_{\Phi,\mathcal{B},\BL}) },
	$$
	whenever $\sigma_{\Phi,\mathcal{B},\BL}\le  t A_{\Phi,\mathcal{B},\BL}$ with $t<1/2$. 
	On the other hand, using Theorem \ref{stabup-real}, we have the trivial estimate
	$$
		\alpha_{\Phi,\mathcal{B},\BL}\le 
		\beta_{\Phi,\mathcal{B},\BL}
		= B_{\Phi,\mathcal{B},\BL}
		\le \sigma_{\Phi,\mathcal{B},\BL}
		\frac{B_{\Phi,\mathcal{B},\BL}}{t A_{\Phi,\mathcal{B},\BL}}
	$$
	whenever $\sigma_{\Phi,\mathcal{B},\BL}\geq  t A_{\Phi,\mathcal{B},\BL}$. Combining these estimates, one can choose 
	$$
		C := \min_{t<1/2} \Big( \frac{4 }{\sqrt{2}(1-2t)},\frac{1}{t}\Big) \frac{\Bp}{\Ap} = 2\,  \frac{\Bp}{\Ap}.
	$$
\end{proof}
The preceding result implies the following immediate corollary.
\begin{corollary}\label{cor:stab-complex-case}
Suppose that $\Phi$ constitutes a frame for $\mathcal{B}$
	and that $\K=\mathbb{C}$. 
	Then 
	%
	$$
		\frac{A_{\Phi,\mathcal{B},\BL}}{2\sigma_{\Phi,\mathcal{B},\BL}}
		\le 
		\tau_{\Phi,\mathcal{B},\BL}.
	$$
	%
\end{corollary}
We note that in the real case, the lower bound in Corollary \ref{cor:stab-real-case} is stronger than the one in Corollary \ref{cor:stab-complex-case} for the general complex case. In Corollary \ref{cor:stab-real-case} we exploit the fact that the quantities $u, v$ determine the distance between $x$ and $y$ which is no longer true in the complex case.
%
%
\subsubsection{Uniform stability always holds in finite dimensions}\label{sec:finite}
Suppose that $\mathcal{B}$ is finite dimensional. We show that
in this case, the condition number $\tau_{\Phi,\mathcal{B},\BL}$
is always finite. In the setting of discrete frames for Hilbert spaces this 
result has been shown in \cite{cahill}.
\begin{theorem}
	Suppose that $\mathcal{B}$ is finite dimensional ($\K$ can be either $\mathbb{R}$ or $\mathbb{C}$). Assume further that $\Phi$ constitutes a frame for $\mathcal{B}$ with associated Banach space $\BL$, and that $\mathcal{A}_\Phi$ is injective.  
	Then, 
	$$
		\tau_{\Phi,\mathcal{B},\BL}<\infty.
	$$
\end{theorem}
\begin{proof}
	We need to study uniform continuity of 
	the map $\mathcal{A}_\Phi^{-1}$. Since this mapping is scaling-invariant, we may without loss of generality only consider
	the mapping $\mathcal{A}_\Phi^{-1}$, restricted to 
	$\mathcal{A}_\Phi(\mathcal{B})\cap B_1$, where $B_1$ denotes
	the unit ball in $\BL$. The set $\mathcal{A}_\Phi(\mathcal{B})\cap B_1$ is compact since $\mathcal{B}$ is finite-dimensional 
	and by Theorem \ref{thm:weakstab} the mapping 
	$\mathcal{A}_\Phi^{-1}$ is continuous on this set.
	Therefore, it is uniformly continuous and this proves the claim. 
\end{proof}
%

%
%
%
%

%
\subsubsection{The SCP can never hold in infinite dimensions}\label{sec:infinite}
We have seen that the SCP is a necessary condition for the stability of the phase retrieval problem \ref{def-problem}, both in the real and the complex case. In the Hilbert space setting and for discrete frames, it has been shown in \cite{cahill}, that phase retrieval can never be done stably if the space is infinite-dimensional. One might ask whether increasing the redundancy of a frame might help. We show that there is unfortunately no hope. More precisely, the following results state that in our more general Banach space setting and even for continuous frames, the SCP can never hold. To prove this, we first need a result, that is interesting in its own right:

\begin{theorem}\label{noframe}
Suppose $\B$ is an infinite-dimensional Banach space. Then, for any compact set $\Lambda$ and any system $\Phi = (\varphi_\lambda)_{\lsub} \subset \B'$, $\Phi$ cannot constitute a frame for $\B$. In particular, for $\BL$ an admissible Banach space associated with $\B$, there is no positive lower bound $A_{\Phi,\B,\BL}>0$ such that
\begin{equation}\label{compact-indexset}
A_{\Phi,\B,\BL} \|x\|_\B \leq \|[x,\Phi]\|_\BL
\end{equation}
for all $x \in \B$.
\end{theorem}

\begin{proof}
We argue by contradiction. Suppose that for a compact set $\Lambda$ and $\Phi = (\varphi_\lambda)_{\lsub}\subset \B'$ a frame for $\B$, there exists $A_{\Phi,\B,\BL}>0$ that satisfies (\ref{compact-indexset}). Let $\nu>0$ be arbitrary. 
Since the mapping $\lambda\mapsto \varphi_\lambda$ is continuous there exists, for every $\lambda\in \Lambda$ an open set $U_\lambda$ such that 
\begin{equation}\label{C-lambda}
	\|\varphi_\omega -\varphi_\lambda\|_{\B'} < \nu\quad \mbox{for all }
	\omega \in U_\lambda. 
\end{equation}
The $U_\lambda$'s form an open cover of the compact set $\Lambda$. Therefore, already finitely many $U_\lambda$'s cover $\Lambda$. We denote them by $U_{\lambda_1},\dots , U_{\lambda_N}$. Without loss of generality we can assume that these sets have "small overlaps" in the sense that 
$$
	\| \sum_{i=1}^N \chi_{U_{\lambda_i}} \|_{\BL} \leq k \| \chi_\Lambda\|_\BL 
$$

for some constant $k$ independent of $N$. Let $c_\Lambda:=k \| \chi_\Lambda\|_\BL$.

By the frame condition of $\B$ and the solidity of $\BL$ we have
\begin{align}
	A_{\Phi,\mathcal{B},\BL}\|x\|_\B&\le\|[x,\Phi] \|_\BL =\||[x,\Phi]| \|_\BL \nonumber \\
	& \le\| \sum_{i=1}^N |[x,\Phi_{U_{\lambda_i}}]| \|_\BL  \nonumber \\
	& \le \| \sum_{i=1}^N (|[x,\varphi_{\lambda_i}]| \chi_{U_{\lambda_i}}(\lambda))_\lsub \|_\BL + \| \sum_{i=1}^N (|[x,\varphi_\lambda-\varphi_{\lambda_i}]| \chi_{U_{\lambda_i}}(\lambda))_\lsub \|_\BL \nonumber \\
	& \le  \sum_{i=1}^N |[x,\varphi_{\lambda_i}]| \| \chi_{U_{\lambda_i}} \|_\BL +  \| \sum_{i=1}^N (|[x,\varphi_\lambda-\varphi_{\lambda_i}]| \chi_{U_{\lambda_i}}(\lambda))_\lsub \|_\BL \label{rhs}
\end{align}
%
%
%
We note that for $\lambda \in U_{\lambda_i}$, by Equation (\ref{C-lambda})
$$
	|[x,\varphi_\lambda-\varphi_{\lambda_i}]| \le \nu \|x\|_\B.
$$
Together with the solidity of $\BL$ and our assumption on the $U_{\lambda_i}$'s, we can rewrite the second term of the RHS in (\ref{rhs}) to
$$
	 \| \sum_{i=1}^N (|[x,\varphi_\lambda-\varphi_{\lambda_i}]| \chi_{U_{\lambda_i}}(\lambda))_\lsub \|_\BL \leq c_\Lambda \nu \|x\|_\B.
$$
%
In summary, writing
$$
	\omega_i:= \| \chi_{U_{\lambda_i}}\|_\BL,
$$
%
we obtain
$$
A_{\Phi,\mathcal{B},\BL}\|x\|_\B \le \sum_{i=1}^N \omega_i |[x,\varphi_{\lambda_i}]|+ c_\Lambda \nu  \|x\|_\B,
$$
or
\begin{equation}\label{finite-frame}
	(A_{\Phi,\mathcal{B},\BL}-c_\Lambda \nu )\|x\|_\B \le \sum_{i=1}^N \omega_i |[x,\varphi_{\lambda_i}]|.
\end{equation}

Note that since $c_\Lambda$ is independent of $\nu$, we can take $\nu$ sufficiently small, such that $A_{\Phi,\mathcal{B},\BL}-c_\Lambda \nu>0$. Let $\Phi_N=(\varphi_{\lambda_i})_{i=1,\dots,N}$. Due to the finite dimensionality of $\Phi_N$, one can pick $x \in (\mbox{span } \Phi_N)_\perp$ with $\|x\|_\B \neq 0$. This is a contradiction to (\ref{finite-frame}). 
\end{proof}

We are now in the position to state our main result.

\begin{theorem}\label{nostab}
	Suppose that $\mathcal{B}$ is infinite-dimensional and let $\BL$
	be an admissible Banach space associated with $\B$. Then for every frame $\Phi\subset \mathcal{B}'$ of $\B$ with upper frame bound $B_{\Phi,\mathcal{B},\BL}<\infty$ it holds that $\sigma_{\Phi,\mathcal{B},\BL}=0$. Consequently, by Theorem \ref{thmstab}, %
	$$\inf_{\Phi\subset \mathcal{B}', B_{\Phi,\mathcal{B},\BL}<\infty}
	\tau_{\Phi,\mathcal{B},\BL} = \infty.$$
\end{theorem}
\begin{proof}Pick $\varepsilon >0$ arbitrary.
	We will show that there exist a subset $S\subset \Lambda$ and elements $u,v\in \mathcal{B}$ which satisfy $\|u\|_\B=1=\|v\|_\B$, together with 
	\begin{equation}\label{instab1}
		\|[u,\Phi_S]\|_\BL \le \varepsilon,
	\end{equation}
	and
	\begin{equation}\label{instab2}
		\|[v,\Phi_{\Lambda \setminus S}]\|_\BL \le \varepsilon.
	\end{equation}
	In fact, we can pick $u$ to be any arbitrary but fixed function with $\|u\|_\B=1$. 
	The frame property of $\Phi$ implies
	$$
		\|[u,\Phi]\|_\BL \le B_{\Phi,\mathcal{B},\BL}.
	$$
	Since, by Assumption \ref{assump}, $\Lambda$ is $\sigma$-compact, there exists a sequence of compact sets $(\Lambda_n)_{n \in \N}$ such that $\Lambda_n \subset \Lambda_{n+1}$ for all $n \in \N$ and $\bigcup_{n \in \N} \Lambda_n = \Lambda.$ On the other hand, we have assumed that the set of compactly supported functions is dense in $\BL$, so that
	$$
		\| [u,\Phi]-[u,\Phi] \chi_{\Lambda_n}\|_\BL \to 0 \quad \text{ as } n \to \infty.
	$$
	Thus, there exists some $N \in \N$ for which
	$$
		\|[u,\Phi] \chi_{\Lambda \setminus \Lambda_N}\|_\BL \le \varepsilon,
	$$
	and hence (\ref{instab1}) holds with $S=\Lambda \setminus \Lambda_N$.
	
	Since $\Lambda_N$ is compact, by Theorem \ref{noframe}, we can pick $v\in \B$
	which satisfies $\|v\|_\B=1$, and (\ref{instab2}) for $S=\Lambda_N$.
\end{proof}

\begin{remark}
The above theorem shows the inherent instability of phase retrieval regardless of the redundancy of the discrete frame from which the phaseless measurements are drawn. For the reconstruction of a real-valued band-limited function $f$ from unsigned measurements as described in Examples \ref{ex:Cahill} and \ref{ex:Beurling-CP}, our theorem implies that even knowing $|f(x)|$ for all $x \in \mathbb{R}$ would not suffice for stable recovery of $f$. Hence, starting from a discrete set of measurements there is no hope to improve on the stability of the reconstruction problem through oversampling.

We further note that a partial remedy to the instability of phase retrieval in the infinite-dimensional setting is to relax the notion of phase retrieval as introduced in \cite{stable-phase}. There, phase retrieval is studied for the case of magnitude measurements $|V_\varphi f|$ from the Gabor or Cauchy wavelet transform. Instead of reconstructing $\tau f$ for some global phase factor $\tau$ from measurements $|V_\varphi f|$, one seeks to recover $\sum_j \tau_j f_j$, where $f=\sum_j f_j$ with $f_j$ such that $|V_\varphi f_j|$ is concentrated in some region $D_j \subset \mathbb{C}$. This relaxation of allowing different phase factors $\tau_j$ on disjoint regions $D_j$ leads to stability results even for infinite-dimensional phase retrieval.
\end{remark}
\begin{remark}
	Note that our result does \emph{not} imply that for a fixed and finite dimension $d\in \mathbb{N}$ there does
	not exist a representation system $\Phi = (\varphi_i)_{i=1}^N \subset \mathbb{C}^d$ such that the condition number 
	$\tau_{\Phi,\mathbb{C}^d,\mathbb{C}^M}$ is independent of $d$! Indeed, the results of \cite{candes2015phasewirt,candesnear,Bandeira_SavingPhase} suggest that for every dimension $d$ there exists
	a (random) representation system $\Phi_d\subset \mathbb{C}^d$ which allows for stable phase retrieval.
	However, our results show that for any finite dimensional representation system $\Phi = (\varphi_i)_{i=1}^N \subset \mathbb{C}^d$ which arises from projecting an infinite-dimensional problem to finite dimensions, the
	condition number of phase retrieval will blow up as the accuracy of the finite-dimensional approximation increases.  
\end{remark}

\textbf{Acknowledgments.} RA has been supported by an ETH Zurich Postdoctoral Fellowship (cofunded by Marie Curie Actions). We thank Ingrid Daubechies for helpful discussions. PG is grateful to Antti Knowles for inspiring discussions.
\bibliographystyle{abbrv}
\bibliography{phase-retrieval.bib}

\appendix

\section{Relative Compactness in Admissible Banach Spaces}
This appendix provides proofs for results on relative compactness in admissible Banach spaces which are needed in the proof of Theorem \ref{thm:weakstab}. 
Our main result is that, under an additional assumption, 
a subset $\mathfrak{K}\subset \BL$ is relatively compact (meaning that its closure is compact) if and only if
the set of its absolute values $|\mathfrak{K}|:=\{|w|:w\in \mathfrak{K}\}$ is relatively compact.
The additional assumption is as follows.
\begin{definition}[Strong equicontinuity]\label{def:equicont}
	Let $\BL$ be an admissible Banach space and let $\mathfrak{K}\subset \BL$.
	We call $\mathfrak{K}$ \emph{strongly equicontinuous} if for all
	$\varepsilon>0$ and all $\lambda\in \Lambda$ there
	exist open sets $U_\varepsilon(\lambda)\subset \Lambda$ such that
	$$
		\sup_{w\in \mathfrak{K}}|w(\mu) - w(\lambda)|\le \varepsilon
		\quad \mbox{for all } \mu \in U_\varepsilon(\lambda).
	$$
\end{definition}
Before we proceed we record a few facts on relative compactness in complete metric spaces. 
A subset $\mathfrak{K}$ of a complete metric space $X$ is called \emph{totally bounded} if
for all $\delta>0$ it admits a finite covering by balls of radius $\delta$, i.e., there exist $n_\delta \in \N$ and $w_1,\dots , 
w_{n_\delta}\in X$ such that
$$
	\mathfrak{K} \subset \bigcup_{i=1}^{n_{\delta}} B^X_\delta(w_k). 
$$
It is a fundamental fact that relative compactness and total boundedness coincide, see for example \cite{yosida}:
\begin{theorem}\label{thm:totalbound}
	Suppose	 that $X$ is a complete metric space and $\mathfrak{K}\subset X$.
	Then $\mathfrak{K}$ is relatively compact if and only if it is totally bounded.
\end{theorem}
We are now ready to prove the main result of this appendix. The strong equicontinuity assumption could be weakened but the following is sufficient for our purposes.
\begin{theorem}Let $\BL$ be an admissible Banach space.
	Suppose that $\mathfrak{K}\subset \BL$ is 
	strongly equicontinuous.
	Then $\mathfrak{K}$ is relatively compact if and only if
	\begin{itemize}
		\item[(RC1)]$\mathfrak{K}$ is bounded, and
		\item[(RC2)] for every $\varepsilon > 0$ there exists
		a compact set $\Lambda_\varepsilon \subset \Lambda$ such that
		$$
			\sup_{w\in \mathfrak{K}}\|w\chi_{\Lambda\setminus \Lambda_{\varepsilon}}\|_{\BL} \le \varepsilon.
		$$
	\end{itemize}
\end{theorem}
\begin{proof}
	We start with the implication that if $\mathfrak{K}$ is relatively compact and strongly equicontinuous, then (RC1) and (RC2) hold. First of all, clearly $\mathfrak{K}$ is bounded if it is relatively compact. We need to show (RC2). By assumption there exists a finite $\varepsilon/2$-cover
	of $\mathfrak{K}$ which we denote by $K_1,\dots , K_N$. Pick $z_i\in K_i$ for all $i=1,\dots , N$. Then, by the admissibility of $\BL$ there exist compact subsets $\Lambda_1,\dots , \Lambda_N\subset \Lambda$ such that 
		$$
			\sup_{i=1,\dots, N} \|z_i \chi_{\Lambda\setminus \Lambda_i}\|_{\BL} \le \varepsilon/2. 
		$$
		Now denote $\Lambda_\varepsilon:=\bigcup_{i=1}^N \Lambda_i$,
		which is a compact subset of $\Lambda$.
		Let $w\in \mathfrak{K}$ be arbitrary. Then, by assumption there exists
		$i\in \{1,\dots , N\}$ with $\|w-z_i\|_{\BL}\le \varepsilon/2$.
		This implies (by the solidity of $\BL$) that
		$$
			\|w\chi_{\Lambda\setminus \Lambda_{\varepsilon}}\|_{\BL}\le \|w-z_i\|_{\BL}+\|z_i\chi_{\Lambda\setminus \Lambda_{i}}\|_{\BL}\le \varepsilon,
		$$
		which implies (RC2).
		
		Now to the converse statement. 
		Suppose that $\mathfrak{K}$ is strongly equicontinuous and that (RC1) and (RC2) hold true. First note that, by Theorem \ref{thm:totalbound}, relative compactness follows if we can show that for every $\varepsilon>0$ there exists 
		a compact subset $\mathfrak{K}_\varepsilon \subset \BL$ with
		\begin{equation}\label{eq:compapprox}
			\sup_{w\in \mathfrak{K}}\inf_{z\in \mathfrak{K}_{\varepsilon}}\|w-z\|_{\BL}\le \varepsilon,
		\end{equation}
		because this would immediately imply the total boundedness of $\mathfrak{K}$.
		
		Pick $\Lambda_{\varepsilon/2}$ as in (RC2). Consider 
		the set $\mathfrak{K}':= \overline{\{w\chi_{\Lambda_{\varepsilon/2}} :\ w\in \mathfrak{K}\}}$.
		By construction we have that
		$$
			\sup_{w\in \mathfrak{K}}\inf_{z\in \mathfrak{K}'}\|w-z\|_{\BL}\le \varepsilon/2.
		$$
		If $\Lambda$ is discrete we are now finished because in 
		that case the set $\mathfrak{K}'$ is compact. To prove the result for general $\Lambda$ we proceed as follows: 
		
		Let $\nu>0$ be arbitrary and let $U_\nu(\lambda)$ be the sets from Definition \ref{def:equicont}. Consider the sets
		$V_\nu(\lambda):=U_\nu(\lambda)\cap \Lambda_{\varepsilon/2}$, $\lambda \in \Lambda_{\varepsilon/2}$,
		which constitute an open covering of the compact set $\Lambda_{\varepsilon/2}$. By compactness of $\Lambda_{\varepsilon/2}$ there exist $\lambda_1,\dots , \lambda_N\in \Lambda_{\varepsilon/2}$ such that the sets
		$V_\nu(\lambda_i)$, $i=1,\dots , N$ cover $\Lambda_{\varepsilon/2}$. By reducing these sets (and sacrificing the openness assumption) we can further achieve for simplicity that the sets
		$V_{\nu}(\lambda_i)$ constitute a disjoint covering in the
		sense that
		\begin{equation}\label{disjoint-cover}
		\sum_{i=1}^N\chi_{V_\nu(\lambda_i)} = \chi_{\Lambda_{\varepsilon/2}}.
		\end{equation}
		Consider the set
		$$
			\mathfrak{K}'':= \overline{\{\sum_{i=1}^N w(\lambda_i)\chi_{V_\nu(\lambda_i)}:\ w\in \mathfrak{K}' \}}.
		$$
		The set $\mathfrak{K}''$ is contained in a finite dimensional subspace (the span of the indicator functions of the sets $V_\nu(\lambda_i)$) and it is bounded, which can be seen as follows:
		
		It suffices to consider elements in $\mathfrak{K}''$ of the form $\sum_{i=1}^N w(\lambda_i)\chi_{V_\nu(\lambda_i)}$, where $w=w\chi_{\Lambda_{\varepsilon/2}}$ and $w \in \mathfrak{K}$ (since the more general case follows by limiting arguments). We have that
		$$
			\| \sum_{i=1}^N w(\lambda_i)\chi_{V_\nu(\lambda_i)}\|_\BL \leq \| w \chi_{\Lambda_{\varepsilon/2}} - \sum_{i=1}^N w(\lambda_i)\chi_{V_\nu(\lambda_i)}\|_\BL + \| w \chi_{\Lambda_{\varepsilon/2}}\|_\BL.
		$$
		Exploiting (\ref{disjoint-cover}), the strong equicontinuity assumption and the solidity of $\BL$, one further obtains
		\begin{align*}
			\| \sum_{i=1}^N w(\lambda_i)\chi_{V_\nu(\lambda_i)}\|_\BL &\leq \|  \sum_{i=1}^N (w-w(\lambda_i))\chi_{V_\nu(\lambda_i)}\|_\BL + \| w \chi_{\Lambda_{\varepsilon/2}}\|_\BL, \\
			& \leq \nu \|  \chi_{\Lambda_{\varepsilon/2}}\|_\BL + \| w \chi_{\Lambda_{\varepsilon/2}}\|_\BL.
		\end{align*}
		Hence, by assumption (RC1), $\mathfrak{K}''$ is bounded. Altogether, this establishes that $\mathfrak{K}''$ is compact.

		 Now, let $w\in \mathfrak{K}$ be arbitrary.
		 Then we have that
		 $$
		 	\|w-\sum_{i=1}^N w(\lambda_i)\chi_{V_\nu(\lambda_i)}\|_{\BL}\le
		 	\|w-w\chi_{\Lambda_{\varepsilon/2}}\|_\BL
		 	+\|\sum_{i=1}^N(w- w(\lambda_i))\chi_{V_\nu(\lambda_i)}\|_\BL .
		 $$
		 The first term above is bounded by $\varepsilon/2$.
		 By the strong equicontinuity assumption and the solidity the second term is bounded by $\nu \|\chi_{\Lambda_{\varepsilon/2}}\|_\BL$.
		 
		 In summary we get that
		 $$
				\|w-\sum_{i=1}^N w(\lambda_i)\chi_{V_\nu(\lambda_i)}\|_{\BL}
				\le \varepsilon/2 + \nu \|\chi_{\Lambda_{\varepsilon/2}}\|_{\BL}. 
		 $$
		 Note that $\nu$ was arbitrary. By making it sufficiently small we can achieve that 
		 $$
		 	\sup_{w\in \mathfrak{K}}\inf_{z\in \mathfrak{K}''}\|w-z\|_{\BL}\le \varepsilon,
		 $$
		 as desired. This implies (\ref{eq:compapprox}) and the proof is finished.
\end{proof}
The previous result is interesting in its own right; for our purpose the following corollary will be especially useful.
\begin{corollary}\label{cor:relcomp}
	Suppose that $\BL$ is an admissible Banach space and that
	$\mathfrak{K}\subset \BL$ is strongly equicontinuous. Then
	$\mathfrak{K}$ is relatively compact if and only if $|\mathfrak{K}|$
	is relatively compact.
\end{corollary}

\end{document}